\documentclass{amsart}

\usepackage{amsmath,amssymb,amsthm}

\newtheorem{prop}{Proposition}[section]
\newtheorem{thm}[prop]{Theorem}
\newtheorem{lem}[prop]{Lemma}

\theoremstyle{definition}

\newtheorem{ex}[prop]{Example}

\newtheorem{rem}[prop]{Remark}

\newtheorem*{ack}{Acknowledgements}

\def\co{\colon\thinspace}

\newcommand{\ip}{{\,\rule{2.3mm}{.2mm}\rule{.2mm}{2.3mm}\;\, }}

\newcommand{\alst}{\alpha_{\mathrm{st}}}

\newcommand{\C}{\mathbb C}
\newcommand{\CP}{\mathbb{C}\mathrm{P}}

\newcommand{\D}{\mathbb D}

\newcommand{\rmd}{\mathrm d}

\newcommand{\rme}{\mathrm e}

\newcommand{\bfoo}{\mathbf{0}}

\newcommand{\bfp}{\mathbf p}

\newcommand{\bfq}{\mathbf q}

\newcommand{\R}{\mathbb R}
\newcommand{\RP}{\mathbb{R}\mathrm{P}}
\newcommand{\wtR}{\widetilde{R}}

\newcommand{\ttt}{\texttt{t}}

\newcommand{\bfv}{\mathbf v}

\newcommand{\WW}{\mathcal{W}}
\newcommand{\bfw}{\mathbf w}

\newcommand{\xist}{\xi_{\mathrm{st}}}
\newcommand{\lcan}{\lambda_{\mathrm{can}}}
\newcommand{\ocan}{\omega_{\mathrm{can}}}

\DeclareMathOperator{\id}{id}

\begin{document}

\author[M.~D\"orner]{Max D\"orner}
\author[H.~Geiges]{Hansj\"org Geiges}
\address{Mathematisches Institut, Universit\"at zu K\"oln,
Weyertal 86--90, 50931 K\"oln, Germany}
\email{mdoerner@math.uni-koeln.de}
\email{geiges@math.uni-koeln.de}
\author[K.~Zehmisch]{Kai Zehmisch}
\address{Mathematisches Institut, WWU M\"unster, Einsteinstra{\ss}e 62,
48149 M\"unster, Germany}
\email{kai.zehmisch@uni-muenster.de}

\thanks{The research reported in this survey is supported by the
Deutsche Forschungsgemeinschaft (GE 1245/2-1 to H.G.\
and ZE 992/1-1 to K.Z.)\ and is part of a project
in the SFB/TRR 191 `Symplectic Structures in Geometry, Algebra and Dynamics'.}

\title[Finsler geodesics, periodic Reeb orbits, and open books]{Finsler
geodesics, periodic Reeb orbits,\\ and open books}

\date{}

\begin{abstract}
We survey some results on the existence (and non-existence)
of periodic Reeb orbits
on contact manifolds, both in the open and closed case. We place these
statements in the context of Finsler geometry by including a proof
of the folklore theorem that the Finsler geodesic flow can be
interpreted as a Reeb flow. As a mild extension of previous results
we present existence statements on periodic Reeb orbits
on contact manifolds with suitable supporting open books.
\end{abstract}


\maketitle


\section{Introduction}
From a Riemannian perspective, Reeb flows on contact manifolds may be
regarded as a generalisation of the geodesic flow on a Finsler manifold.
It is therefore reasonable to ask to what extent statements
about geodesic flows are instances of more
general facts in Reeb dynamics.

In this brief survey, we are interested in existence and
non-existence statements about periodic Reeb orbits. Recall
that a \emph{contact form} on a $(2n+1)$-dimensional manifold
is a $1$-form $\alpha$ such that $\alpha\wedge(\rmd\alpha)^n$ is
nowhere zero, i.e.\ a volume form. The \emph{Reeb vector field}
$R_{\alpha}$ of $\alpha$ is uniquely defined by the equations
\[ R_{\alpha}\ip\rmd\alpha=0\;\;\;\text{and}\;\;\;\alpha(R_{\alpha})=1,\]
where $\ip$ denotes the interior product.

A (cooriented) \emph{contact structure} $\xi\subset TM$ is the
maximally non-integrable tangent hyperplane field defined as the kernel
of a contact form $\alpha$; this contact form is
unique up to multiplication with a positive function.
A contact manifold $(M,\xi)$ is said to satisfy the
\emph{Weinstein conjecture}~\cite{wein79}, if
\begin{itemize}
\item[(W)] the Reeb vector field of every contact form defining $\xi$ has
a periodic Reeb orbit.
\end{itemize}

For closed manifolds, this conjecture
has been established under various assumptions on
the topology of $M$ or particular properties of~$\xi$
(such as so-called overtwistedness). In its full generality,
the conjecture remains open. There are simple examples of open
manifolds where (W) fails, but recently there has been some
progress on establishing (W) for a number of open contact manifolds.

A stronger requirement is that
\begin{itemize}
\item[(W$^{\circ}$)]
the Reeb vector field of  every contact form defining $\xi$ has
a contractible periodic orbit.
\end{itemize}

Here the period for which the orbit is contractible is not required
to be the minimal one. Even with this relaxation, one cannot
expect (W$^{\circ}$) to hold for all closed contact manifolds.
Take, for instance, the $3$-torus
$S^1\times S^1\times S^1$ with the contact form
$\sin\theta\,\rmd x-\cos\theta\,\rmd y$. The Reeb vector field
$\sin\theta\,\partial_x-\cos\theta\,\partial_y$ has plenty of
periodic Reeb orbits, all of them non-contractible.

The fact that the Finsler geodesic flow is a Reeb flow has been
mentioned in various sources, but we have not been able to
find a reference that includes a complete proof. Therefore, as a service
to the reader, we present such a demonstration in
Section~\ref{section:Finsler}. In Section~\ref{section:dual}
we argue that it is expedient to view Finsler geodesic flows
on the cotangent rather than the tangent bundle, where Reeb flows
appear not as an \emph{ad hoc} interpretation, but
as the natural generalisation of Finsler geodesic flows.

In Section~\ref{section:non-compact} we discuss the existence and
non-existence of periodic Reeb orbits on non-compact manifolds.
The last two sections of this survey deal with the Weinstein conjecture
on closed manifolds. It is shown how information on the topology
of the manifold allows one to prove (W) or (W$^{\circ}$).
In Section~\ref{section:surgery}, the information
on the topology comes from surgery; in Section~\ref{section:books},
from open book decompositions. This last section contains instances
of (W) or (W$^{\circ}$) that have not previously appeared in the literature.
\section{Finsler geodesics as Reeb orbits}
\label{section:Finsler}
It is a folklore result that the Finsler geodesic flow can be viewed as
a Reeb flow, see \cite[Theorem~1.1.2]{bart12} or
\cite{hrsa13}, for instance. Here we present a complete
proof of this fact, inspired by the remark in \cite[p.~6]{bart12}
that one should identify the defining equations for the Reeb
vector field as the Euler--Lagrange equations for the length
functional.

On a manifold $Q$, we write tangent vectors as
$(\bfq,\bfv)\in T_{\bfq}Q$ or simply $\bfv$ if the base point
is clear from the context. Local coordinates
on the tangent bundle $TQ$ are chosen such that
\[ (q^1,\ldots,q^n,v^1,\ldots,v^n)=\Bigl(v^i\frac{\partial}{\partial q^i}
\Bigr)_{(q^1,\ldots,q^n)},\]
with summation convention understood.
A \emph{Finsler metric} on $Q$ is a function
$F\co TQ\rightarrow\R_0^+$ on the tangent bundle with the
following properties:
\begin{itemize}
\item[(i)] $G:=F^2/2$ is smooth outside the zero section $Q\subset TQ$.
\item[(ii)] $F(\bfq,\bfv)=0$ if and only if $\bfv =\bfoo$.
\item[(iii)] $F(\bfq,\rho\bfv)=\rho F(\bfq,\bfv)$ for $\rho\in\R^+$.
\item[(iv)] $G$ is strictly convex, i.e.\ at each point $\bfq\in Q$
the function $\bfv\mapsto G(\bfq,\bfv)$ on $T_{\bfq}Q\setminus\{\bfoo\}$
has a positive definite Hessian matrix.
\end{itemize}
Condition (iv) implies that the sublevel sets of $G$ (or~$F$) are
strictly convex, cf.~\cite[Lemma~5.3.1]{geig08}.

A Riemannian metric $g$ on $Q$ gives rise to a Finsler metric
by setting $F(\bfq,\bfv)=\bigl(g_{\bfq}(\bfv,\bfv)\bigr)^{1/2}$.
In this case $G$ is smooth on all of $TQ$. In fact, this
smoothness property at the zero section
characterises Finsler metrics coming from a Riemannian metric,
see \cite[Lemma~1.4.1]{abpa94}.

The length of a smooth curve $\gamma\co [a,b]\rightarrow Q$ with respect to
the Finsler metric $F$ is defined as
\[ L(\gamma)=\int_a^b F\bigl(\gamma(t),\dot{\gamma}(t)\bigr)\, \rmd t.\]
The critical points of this length functional are the
\emph{Finsler geodesics} when para\-metrised proportional to arc length.
For a characterisation of Finsler geodesics as critical points
of an energy functional see~\cite{rade04a, rade04b}; here the parametrisation
proportional to arc length is automatic.

We do \emph{not} require the Finsler metric to be \emph{reversible},
i.e.\ we do not impose the condition $F(\bfq,-\bfv)=F(\bfq,\bfv)$.
Consequently, the orientation of a curve is essential for its
being a geodesic. Nonetheless, as in the Riemannian setting,
geodesics are defined by a geodesic vector field on the
tangent bundle, defining a directed geodesic flow.
Thus, through every point in $TQ\setminus Q$ there
is a unique directed curve projecting to a geodesic;
see \cite[Section~1.5]{abpa94}.

In order to relate Finsler geodesics to Reeb flows, we rewrite this
length functional in terms of the so-called \emph{Hilbert form}.
In local coordinates on $TQ\setminus Q$, this $1$-form is defined as
\[ \beta:=\frac{\partial F}{\partial v^i}\,\rmd q^i.\]
Given different local coordinates $(x^j,w^j)$, we have
\[ v^i=\frac{\partial q^i}{\partial x^j}\, w^j,\]
whence
\[ \frac{\partial v^i}{\partial w^j}=\frac{\partial q^i}{\partial x^j}.\]
This implies
\[ \frac{\partial F}{\partial w^j}\,\rmd x^j=
\frac{\partial F}{\partial v^i}\,\frac{\partial v^i}{\partial w^j}\,\rmd x^j
=\frac{\partial F}{\partial v^i}\,\frac{\partial q^i}{\partial x^j}
\,\rmd x^j=\frac{\partial F}{\partial v^i}\,\rmd q^i,\]
which shows that the Hilbert form
$\beta$ is globally defined on $TQ\setminus Q$.

Observe that by Euler's theorem on homogeneous functions and
the homogeneity condition (iii) we
have $v^i(\partial F/\partial v^i)=F$. This allows one to rewrite
the arc length integral as an integral of the Hilbert form, where
$\Gamma=(\gamma,\dot{\gamma})$ is the curve on $TQ\setminus Q$
defined by a regular curve~$\gamma$ on $Q$:
\[ L(\gamma)=\int_a^b F\bigl(\gamma(t),\dot{\gamma}(t)\bigr)\,\rmd t=
\int_a^b\frac{\partial F}{\partial v^i}
\bigl(\gamma(t),\dot{\gamma}(t)\bigr)\,\dot{\gamma}^i(t)\,\rmd t=
\int_{\Gamma}\beta.\]

It is well known that the Hilbert form $\beta$ is a contact
form on the level sets of $F$ (or on the projectivised tangent bundle),
see~\cite[Chapter~8]{ccl99}. However, for the relation with
symplectic geometry, it seems to be more convenient to work with
the $1$-form
\[ \alpha:=F\beta=\frac{\partial G}{\partial v^i}\,
\rmd q^i.\]
We are going to show that $\alpha$ is a contact form on each level set
of~$F$, and hence so is $\beta$.

\begin{prop}
The $2$-form $\rmd\alpha$ is a symplectic form on $TQ\setminus Q$.
\end{prop}

\begin{proof}
We compute
\[ \rmd\alpha=\frac{\partial^2G}{\partial v^i\,\partial v^j}\,
\rmd v^j\wedge\rmd q^i+
\frac{\partial^2G}{\partial v^i\,\partial q^j}\,
\rmd q^j\wedge\rmd q^i.\]
For the $n$-fold wedge product $(\rmd\alpha)^n$, only the
first summand is relevant, and $\rmd\alpha$ being symplectic is
seen to be a consequence of the convexity condition~(iv).
\end{proof}

A computation as above shows that the fibrewise radial vector field $V:=v^i
\frac{\partial}{\partial v^i}$ is independent of the choice of local
coordinates. This also follows from the next lemma.

\begin{lem}
The radial vector field $V$ satisfies $V\ip\rmd\alpha=\alpha$
and $V\ip\rmd\beta=0$.
\end{lem}

\begin{proof}
Euler's theorem on homogeneous functions gives us
\[ v^k\,\frac{\partial G}{\partial v^k}=2G.\]
Differentiating this equation with respect to $v^i$, we find
\[ v^k\,\frac{\partial^2G}{\partial v^i\,\partial v^k}=
\frac{\partial G}{\partial v^i}.\]
Now we compute
\begin{eqnarray*}
V\ip\rmd\alpha & = & v^k\,\frac{\partial}{\partial v^k}\ip
                     \frac{\partial^2G}{\partial v^i\,\partial v^j}
                     \,\rmd v^j\wedge\rmd q^i\\
 & = & v^k\,\frac{\partial^2G}{\partial v^i\partial v^k}\,
                     \rmd q^i\\
 & = & \frac{\partial G}{\partial v^i}\,\rmd q^i
       \: = \; \alpha.
\end{eqnarray*}

Based on the corresponding identity
\begin{equation}
\label{eqn:F-hom}
v^k\,\frac{\partial^2F}{\partial v^i\,\partial v^k}=0
\end{equation}
for $F$, the computation showing that $V\ip\rmd\beta=0$ is similar.
\end{proof}

With the Cartan formula for the Lie derivative
it follows that $V$ is a \emph{Liouville vector field} with respect
to the symplectic form $\rmd\alpha$, that is, $L_{V}\rmd\alpha=
\rmd\alpha$. Moreover, the homogeneity condition (iii) guarantees
(by Euler's theorem) that $V$ is transverse to the level sets of~$F$.
This entails that $\alpha$ induces a contact form on each level set,
cf.~\cite[Lemma~1.4.5]{geig08}.

Write $M$ for the level set $F^{-1}(1)$. On this level set we have
the contact form $\alpha|_{TM}=\beta|_{TM}$ with Reeb vector field~$R$.
Since $V\ip\rmd\beta=0$, the
kernel of the $2$-form $\rmd\beta|M$ is spanned by $R$ and~$V$.
Also, we have $\beta(V)=0$.
Hence, the Reeb vector field $R$ of $\beta|_{TM}$ can equivalently be
defined as follows. Let $\wtR$ be a section of $T(TQ)$ along $M$
such that
\begin{equation}
\label{eqn:wtR}
\wtR\ip\rmd\beta=0\;\;\text{on}\;\; T(TQ)|_M\;\;\;\text{and}\;\;\
\beta(\wtR)=1.
\end{equation}
Then $R$ equals the projection of $\wtR$ onto $TM$ along the line bundle
$\langle V\rangle$.

\begin{lem}
A local solution $\wtR$ of (\ref{eqn:wtR}) is given by
\[ \wtR=v^i\,\frac{\partial}{\partial q^i}+
a^i\,\frac{\partial}{\partial v^i},\]
where $(a^1,\ldots,a^n)$ is a solution of the equations
\begin{equation}
\label{eqn:a^i}
v^k\,\frac{\partial^2F}{\partial v^i\,\partial q^k}-
v^k\,\frac{\partial^2F}{\partial v^k\,\partial q^i}+
a^k\,\frac{\partial^2F}{\partial v^i\,\partial v^k}=0,\;\;\; i=1,\ldots,n.
\end{equation}
\end{lem}

\begin{proof}
We first convince ourselves that (\ref{eqn:a^i}) has
a solution. Along $M=F^{-1}(1)$ we have
\[ \frac{\partial^2G}{\partial v^i\,\partial v^k}=
\frac{\partial^2F}{\partial v^i\,\partial v^k}+
\frac{\partial F}{\partial v^i}\,\frac{\partial F}{\partial v^k}.\]
Write $F_i=\partial F/\partial v^i$ for brevity, similar for
higher derivatives. The matrix
\[ (F_iF_k)=(F_1,\ldots,F_n)^{\ttt}(F_1,\ldots,F_n)\]
has rank~$1$. The matrix $(G_{ik})$ has rank $n$ by the
convexity condition~(iv). It follows that $(F_{ik})$
has rank at least equal to~$n-1$. Since the vector
$(v^1,\ldots,v^n)$ lies in the kernel of this last matrix,
its rank equals $n-1$, and the image of this symmetric matrix
is the subspace of $\R^n$ orthogonal to $(v^1,\ldots,v^n)$.
It follows that (\ref{eqn:a^i}) indeed has a solution
$(a^1,\ldots,a^n)$, unique up to adding multiples of
$(v^1,\ldots,v^n)$.

Given such a solution, let $\wtR$ be as defined in the lemma.
We first check
\[ \beta(\wtR)=v^i\,\frac{\partial F}{\partial v^i}=F=1\;\text{on $M$}.\]
With the help of (\ref{eqn:F-hom}) and (\ref{eqn:a^i})
one easily verifies the first condition in~(\ref{eqn:wtR}).
\end{proof}

Consequently, the Reeb vector field $R$ is of the form
\[ R=v^i\,\frac{\partial}{\partial q^i}+b^i\,\frac{\partial}{\partial v^i}.\]
This means that Reeb orbits in $M\subset TQ$
are of the form $(\gamma,\dot{\gamma})$,
where $\gamma$ is a unit speed trajectory in~$Q$.

\begin{prop}
The unit speed geodesics on $Q$ with respect to the Finsler metric $F$
are precisely the projections of Reeb orbits on~$M$.
\end{prop}

\begin{proof}
Let $\gamma\co [a,b]\rightarrow Q$ be a unit speed curve, and write
$\Gamma=(\gamma,\dot{\gamma})$ for the corresponding curve
in~$M$.
Consider a variation $\Gamma_s=(\gamma_s,\dot{\gamma}_s)$, $s\in
(-\varepsilon,\varepsilon)$,
of curves in $TQ\setminus Q$, fixed near the endpoints.
Write
\[ X=\frac{\rmd\Gamma_s}{\rmd s}|_{s=0}\]
for the variational vector field at $s=0$. Then
\[ \frac{\rmd}{\rmd s}|_{s=0}L(\gamma_s)=
\frac{\rmd}{\rmd s}|_{s=0}\int_a^b\Gamma_s^*\beta
=\int_{\Gamma}\bigl(\rmd(\beta(X))+X\ip\rmd\beta\bigr);
\]
see \cite[Lemma~B.1]{geig08} for the last equality.
The integral over the first summand vanishes since the
variation is fixed near the end points. So the condition for $\gamma$
to be a geodesic, that is, the vanishing of this first variation,
becomes
\[ \int_a^b\rmd\beta(X,\dot{\Gamma})\,\rmd t=0.\]

It would be precipitate to conclude that this forces $\dot{\Gamma}$
to equal (a multiple of) the Reeb vector field, since only a subclass of
vector fields arises as variational vector fields~$X$. 
However, for Reeb trajectories this first variation certainly
vanishes, even under general variations. Moreover, we have seen
that Reeb trajectories are of the form $(\gamma,\dot{\gamma})$,
with $\gamma$ of unit speed. This implies that every Reeb orbit
does indeed project to a Finsler geodesic.

Conversely, through every point $(\bfq,\bfv)\in M\subset TQ$
there is a unique curve projecting to a unit speed geodesic, so
this curve coincides with the Reeb orbit through that point.
\end{proof}

\begin{rem}
For a particular case of this result, the geodesic flow of a
\emph{Riemannian} manifold, a proof can be found in
\cite[Section~1.5]{geig08}. The proof
above is simpler, since it avoids computations in special local
coordinates.

A proof of the general case, using a different line of reasoning
and written in Portuguese, can be found in~\cite{hrsa13a}, where,
in Teorema~4.4.10, the Hilbert form is
identified with the pull-back of the canonical Liouville $1$-form
on the unit cotangent bundle, and the geodesic spray with
the Reeb vector field.
\end{rem}
\section{The dual viewpoint}
\label{section:dual}
In this section we expand a little on that last remark
and explain why Finsler geodesic flows have a more natural interpretation
in the cotangent bundle. We thank Felix Schlenk~\cite{schl16}
and the referees for useful suggestions regarding this issue.

On the cotangent bundle $T^*Q$ we choose local coordinates such that
\[ (q^1,\ldots,q^n,p_1,\ldots,p_n)=
\bigl(p_i\,\rmd q^i\bigr)_{(q^1,\ldots,q^n)}.\]
In these coordinates, and the adapted coordinates on $TQ$ as
in the preceding section, the Legendre transformation
$\mathcal{L}\co TQ\setminus Q\rightarrow T^*Q\setminus Q$
defined by the Lagrange function $G\co TQ\rightarrow\R$ is described by
\[ (q^1,\ldots,q^n,v^1,\ldots,v^n)\longmapsto
\Bigl(q^1,\ldots,q^n,p_1=\frac{\partial G}{\partial v^1},\ldots,
p_n=\frac{\partial G}{\partial v^n}\Bigl),\]
see \cite[8.3]{lima87}. The homogeneity and strong convexity of $G$
ensure that $\mathcal{L}$ is a diffeomorphism.
Under this diffeomorphism, the canonical Liouville form $\lcan=p_i\,\rmd q^i$
on $T^*Q$ pulls back to the contact form on $TQ$ that we called~$\alpha$.

\begin{rem}
In the Riemannian case, when $G=g_{ij}v^iv^j/2$, we have
$\partial G/\partial v^i=g_{ij}v^j$, so the Legendre transformation
coincides with the identification of $TQ$ with $T^*Q$ given by the metric.
For this case see also \cite[Theorem~1.5.2]{geig08}.
\end{rem}

The strict convexity of the unit ball $B_{\bfq}\subset T_{\bfq}Q$
with respect to~$F$,
\[ B_{\bfq}:=\bigl\{\bfv\in T_{\bfq}Q\co G(\bfq,\bfv)\leq 1/2\bigr\},\]
implies that for  $\bfv\in\partial B_{\bfq}$ and $\bfw$ a tangent vector
to the fibre $T_{\bfq}Q$ at $\bfv\in T_{\bfq}Q$ we have
\[ \frac{\partial G}{\partial v^i}\,w^i\leq
\frac{\partial G}{\partial v^i}\,v^i=2G=1,\]
with equality only for $\bfw=\bfv$ (interpreting $\bfw$ as an element of
$T_{\bfq}Q$). So the polar body $B_{\bfq}^*:=
\mathcal{L}(B_{\bfq})\subset T_{\bfq}^*Q$ is the set
\[ B_{\bfq}^*=\bigl\{\bfp\in T_{\bfq}^*Q\co
p_iv^i\leq 1\;\text{for all}\; \bfv\in B_{\bfq}\bigr\}.\]
Since $G$ is fibrewise homogeneous of degree~$2$, the Hamiltonian
function $G^*\co T^*Q\rightarrow\R$ corresponding to $G$ is
simply given by $G^*=G\circ\mathcal{L}^{-1}$, see~\cite[8.6]{lima87}.
In other words, $G^*$ takes the value $1/2$ on the boundary of $B_{\bfq}^*$
and is likewise fibrewise homogeneous of degree~$2$.

It now becomes perfectly obvious why the Finsler geodesic flow
is a Reeb flow. The Finsler geodesics are the solutions of the
Lagrangian system defined by the function $G$ on $TQ$, and
they transform under $\mathcal{L}$ to the solutions of the
Hamiltonian system on $T^*Q$ (equipped with the canonical symplectic form
$\ocan=\rmd p_i\wedge\rmd q^i$) defined by~$G^*$, see
\cite[Chapter~10]{lima87}. For the
following lemma cf.~\cite[Lemma~4.2]{fls15} and
\cite[Lemma~1.4.10]{geig08}.

\begin{lem}
The Hamiltonian flow of $G^*$ on the unit cotangent bundle
$M:=ST^*Q=\{G^*=1/2\}$ coincides with the Reeb flow
of the contact form $\lcan|_{TM}$ given by the restriction of the
Liouville form.
\end{lem}

\begin{proof}
Both the Hamiltonian vector field $X_{G^*}|_M$ and the Reeb vector
field of $\lcan|_{TM}$ span the kernel of $\ocan|_{TM}$.
Write $Y=p_i\frac{\partial}{\partial p_i}$ for the fibrewise
radial vector field on $T^*Q$, so that $\lcan=Y\ip\ocan$. From
the computation
\[ \lcan(X_{G^*})=\ocan(Y,X_{G^*})=\rmd G^*(Y)=2G^*=1\;\text{on $M$},\]
where we have used the homogeneity of~$G^*$, we conclude
that $X_{G^*}|_M$ coincides with the Reeb vector field.
\end{proof}

Rescaling the contact form $\lcan|_{TM}$ by a positive function is
equivalent to taking the induced contact form on a fibrewise starshaped
hypersurface in $T^*Q$. In this way,
a large class of Reeb flows arises as a natural generalisation
of the Finsler \emph{cogeodesic} flow. 

Finsler (co-)metrics come from fibrewise strictly convex
hypersurfaces enclosing the zero section in $TQ$ and $T^*Q$, respectively.
Both viewpoints are equivalent under the Legendre transformation.
In the cotangent bundle, there is a natural extension of Finsler
cogeodesic flows to the rich theory of Reeb flows on fibrewise
starshaped hypersurfaces enclosing the zero section, see~\cite{fls15}.
Here the Legendre transformation
breaks down, so there is no correspondence with metrics on $Q$
defined by fibrewise starshaped
hypersurfaces in the tangent bundle.

Indeed, as pointed out by Busemann~\cite[p.~83]{buse55},
if one drops the condition that $\{F=1\}\cap T_{\bfq}Q$ be convex
in $T_{\bfq}Q$ for all~$\bfq\in Q$, the resulting metric geometry
has some awkward features. Given any function
$F\co TQ\rightarrow\R_0^+$ satisfying conditions (i) to (iii)
as in the preceding section, one can define the $F$-length
$L(\gamma)=\int F(\gamma,\dot{\gamma})$ of a (piecewise)
$C^1$-curve $\gamma$ in $Q$.
The distance between two points $\bfq_0,\bfq_1$ in $Q$
is naturally set to be the infimum over the $F$-length of all
$C^1$-curves joining $\bfq_0$ with~$\bfq_1$. With the help
of this distance function one can then talk about rectifiable
curves and their arc length obtained by polygonal approximation.
However, this arc length will coincide
with the $F$-length on all $C^1$-curves only if the
unit spheres in $T_{\bfq}Q$ with respect to $F$ are convex for
all $\bfq\in Q$.

To summarise, only when viewing Finsler geometry on the
cotangent rather than the tangent bundle do we perceive
a natural and fruitful generalisation of Finsler geodesic flows,
first to Reeb flows on fibrewise starshaped hypersurfaces, thence to
general contact manifolds.
\section{Existence vs.\ nonexistence}
\label{section:non-compact}
On non-compact contact manifolds, the Reeb vector field will not,
in general, have any periodic orbits. Write
\[ \alst=\rmd z+\frac{1}{2}\sum_{j=1}^n (x_j\,\rmd y_j-y_j\,\rmd x_j)\]
for the standard contact form on $\R^{2n+1}$, and $\xist=\ker\alst$
for the standard contact structure. The Reeb vector field
of $\alst$ is the coordinate vector field~$\partial_z$.

We also write $\xist=\ker\alst$ for the standard contact structure
and contact form, respectively,
on the sphere $S^{2n+1}\subset\R^{2n+2}$, defined by
\[ \alst=\sum_{j=1}^{n+1}(x_j\,\rmd y_j-y_j\,\rmd x_j).\]
This notational duplication is justified by the fact that
the two contact structures $\xist$ on $\R^{2n+1}$ and
on the complement of a point in $S^{2n+1}$, respectively,
are diffeomorphic, see~\cite[Proposition~2.1.8]{geig08}.
The Reeb vector field of $\alst$ on $S^{2n+1}$ is given by
\[ \sum_{j=1}^{n+1}(x_j\partial_{y_j}-y_j\partial_{x_j}),\]
so all Reeb orbits are periodic, and they define
the generalised Hopf fibration $S^1\hookrightarrow S^{2n+1}\rightarrow
\CP^n$.

Going back to $(\R^{2n+1},\alst)$,
it is natural to ask whether local changes in the topology
or the contact form force the existence of periodic Reeb orbits.
One very simple way to make a local change in the contact form,
but no change in the topology, that will produce a periodic Reeb
orbit is to take a contact connected sum with $(S^{2n+1},\alst)$
inside a small ball in $\R^{2n+1}$; see
also the next section. For the definition of contact surgery we refer
the reader to \cite{wein91} and \cite[Chapter~6]{geig08}.

An essential tool for finding or excluding periodic Reeb orbits
are holomorphic discs. To start with the most simple example,
consider the cylinder
\[ Z=D^2\times\R=\{x^2+y^2\leq 1\} \subset \R^3.\] 
Let $\alpha$ be a contact form on $\R^3$ that coincides with
$\alst$ outside a small (not necessarily round)
ball $B$ contained in~$Z$. Equip the
symplectisation $(\R\times\R^3, \omega:=\rmd(\rme^t\alpha))$ with
an almost complex structure $J$ which is tamed by $\omega$,
that is, $\omega(X,JY)$ defines a $J$-invariant Riemannian metric,
and which satisfies $J(\partial_t)=R_{\alpha}$. Away from $\R\times B$
we take $J$ to be the standard complex structure $J(\partial_x)=\partial_y$,
$J(\partial_t)=\partial_z$. Then the domain
$(-\infty,0]\times Z$ has a piecewise smooth strictly
pseudoconvex boundary.

One now studies holomorphic discs
\[ (a,f)\co \D^2\longrightarrow
\R\times Z\]
whose boundary map $\partial\D^2\rightarrow\{0\}\times\partial Z$ is of
degree one.
Away from $\R\times B$, we have the obvious holomorphic
discs $\{0\}\times D^2\times\{z\}$.
Standard methods on pseudoholomorphic curves in
symplectisations~\cite{hofe93} and the filling with holomorphic
discs \cite{elia90,grom85} then show the following alternative: either
\begin{itemize}
\item[(i)] the foliation by standard discs away from $\R\times B$ extends to
a foliation of an embedded cylinder in $(-\infty,0]\times Z$, or
\item[(ii)] there is a sequence of holomorphic discs whose
gradient explodes.
\end{itemize}

This particular situation was investigated by Eliashberg--Hofer~\cite{elho94}.
In case (i), they show that the foliation by holomorphic discs
projects to a foliation of the cylinder $\{0\}\times Z$ by
embedded discs. By the definition of~$J$, the Reeb orbits will
be transverse to these discs, so there can be neither a periodic
nor a trapped Reeb orbit (i.e.\ an orbit that enters the ball $B$
but does not leave it). In fact, Eliashberg--Hofer show that $\alpha$
is then diffeomorphic to~$\alst$. In case (ii), again by the arguments
of~\cite{hofe93}, one finds a periodic Reeb orbit.

The fact that $\alpha$ is diffeomorphic to $\alst$ in the absence of
periodic Reeb orbits can be interpreted as a global Darboux theorem.
Another way to phrase this conclusion is as follows.

\begin{thm}[Eliashberg--Hofer]
\label{thm:EH}
Let $\alpha$ be a contact form on $\R^3$ that equals the standard form
$\alst$ outside a compact set. If the Reeb vector field of $\alpha$
has a trapped orbit, then it also has a periodic orbit.\qed
\end{thm}

Although, in higher dimensions, a foliation by holomorphic discs
is less restrictive for the Reeb dynamics, it was expected
that an analogous result might hold there.
This was refuted in~\cite{grz14}, even under the additional
restriction that the contact structure $\xist$ remain unchanged.

\begin{thm}
There is a deformation of $\alst$ on $\R^{2n+1}$, $n\geq 2$,
into a contact form $\alpha=h\alst$, with $h\co\R^{2n+1}\rightarrow\R^+$
and $h-1$ compactly supported, such that $R_{\alpha}$ has trapped, but
no periodic orbits.
\qed
\end{thm}

The idea of the proof is to prescribe a certain dynamics, specifically,
an irrational flow on a Clifford torus in $\C^n\equiv\R^{2n}\times\{0\}
\subset\R^{2n+1}$ acting as a trap for orbits, and then to realise
this dynamics as a contact Hamiltonian flow positively
transverse to $\xist$ by translating
these prescriptions into properties of the Hamiltonian function;
the transversality condition guarantees that the Hamiltonian vector field
is the Reeb vector field of a rescaled contact form.
A similar dynamics for Riemannian geodesics was constructed in
\cite{baro12}: there are Riemannian metrics on $\R^n$, $n\geq 4$,
equal to the Euclidean metric outside a compact set,
with bounded geodesics but no periodic ones.

Eliashberg and Hofer also discuss the effect of changing the
topology inside $B$ rather than just the contact form.
Write $\hat{Z}$ for the cylinder with such a change of topology
performed inside a ball $B\subset Z$.
Here the same alternative as above holds, and
the foliation by holomorphic discs in the absence of periodic
Reeb orbits, i.e.\ alternative (i), is shown in \cite{elho94}
to prevent any non-trivial
topology. In contrast with the previous two theorems,
this Hamiltonian characterisation of the ball extends to higher dimensions.
To this end, one introduces a moduli space $\WW$ of holomorphic discs
as before and studies the evaluation map
$\WW\times\D\rightarrow\hat{Z}$, $\bigl((a,f),z\bigr)\mapsto f(z)$.
When there are no periodic Reeb orbits, this map will be proper and
surjective.  By using degree-theoretic methods and the $h$-cobordism
theorem, one arrives at the following result, see~\cite{geze16},
extending \cite{elho94} from three to higher dimensions.
In alternative (ii), one always finds a
contractible Reeb orbit whose period can be estimated by
an energy integral; this accounts for the quantitative statement.

\begin{thm}
\label{thm:ball}
Assume that $(M,\alpha)$ is a compact contact manifold with boundary $S^{2n}$,
such that, near the boundary, the contact form $\alpha$ looks like that of
a ball $B\subset (Z,\alst)$. If $R_{\alpha}$ has no contractible
periodic orbits of period smaller than~$\pi$, the manifold $M$
is diffeomorphic to a ball.
\qed
\end{thm}

For other results concerning the existence of periodic Reeb orbits on
non-compact manifolds, including cotangent bundles
over a non-compact base, see \cite{bpv09, bprv, suze16}.
\section{Contact surgery}
\label{section:surgery}
We now turn our attention to \emph{closed} contact manifolds $(M,\xi)$.
As indicated earlier, the methods introduced by Hofer~\cite{hofe93}
--- when the
conditions for their applicability are satisfied --- in their basic form
show the existence of a \emph{contractible} periodic Reeb orbit.
This places limitations on these methods, as there
are obvious examples of contact manifolds without
contractible periodic Reeb orbits, as mentioned in the introduction.

Theorem~\ref{thm:ball} can be read as an example where (W$^{\circ}$)
holds in the presence of non-trivial topology.
In the same vein, one may ask if (W) or (W$^{\circ}$) holds for any
contact structure on a closed manifold of sufficiently
complicated topology. One way to create non-trivial topology
is by performing surgery. The most simple type of surgery
is forming the connected sum $M_1\#M_2$ of two (connected) manifolds
$M_1,M_2$ of the same dimension~$m$: remove an open $m$-disc from
either manifold,
and glue in a tube $S^{m-1}\times [-1,1]$ to connect the two
manifolds. If neither manifold was a sphere, this creates non-trivial
topology. Then, in particular, the \emph{belt sphere} $S^{m-1}\times\{0\}$
does not bound a ball in $M_1\#M_2$.

In \cite{hofe93}, Hofer proved (W$^{\circ}$) for
closed $3$-dimensional contact manifolds $(M,\xi)$
with $\xi$ overtwisted or $M$ having non-trivial
second homotopy group. According to Eliashberg's
classification \cite{elia92} of contact structures on the $3$-sphere,
all contact structures on $S^3$ except the standard one are
overtwisted. These two results can be combined as follows.

\begin{thm}[Eliashberg, Hofer]
\label{thm:sumEH}
Let $(M,\xi)=(M_1,\xi_1)\#(M_2,\xi_2)$ be the contact connected sum
of two closed, connected contact $3$-manifolds.
If $(M,\xi)$ does not satisfy $\mathrm{(W^{\circ})}$, then one of the
summands $(M_i,\xi_i)$ is contactomorphic to $(S^3,\xist)$.
\qed
\end{thm}

In other words, (W$^{\circ}$) and \emph{a fortiori} (W) holds for all
non-trivial contact connected sums. In fact, (W) is known
to hold for all closed contact $3$-manifolds by the
work of Taubes~\cite{taub07}, whose proof uses Seiberg--Witten theory;
see~\cite{hutc10} for an exposition of this proof.

In \cite{geze16a} we extended Theorem~\ref{thm:sumEH} to higher
dimensions, under some topological conditions.

\begin{thm}
\label{thm:sumGZ}
Let $(M,\xi)=(M_1,\xi_1)\#(M_2,\xi_2)$ be the contact connected sum
of two closed, connected contact manifolds of dimension
$2n+1\geq 5$. Assume further that
\begin{itemize}
\item[(i)] $M$ is simply connected and has torsion-free homology, or
\item[(ii)] $M$ is not simply connected.
\end{itemize}
Then, if $(M,\xi)$ does not satisfy $\mathrm{(W^{\circ})}$,
one of the summands $M_i$ is homeomorphic to~$S^{2n+1}$.
\qed
\end{thm}

This theorem can potentially be used
as a contact-geometric primality test for manifolds.
Results of this type may even have quite concrete applications in physics.
For instance, Albers et al.~\cite{afkp12}, in their study of the
planar circular restricted $3$-body problem, cf.~\cite{geig16},
show that the energy hypersurface for levels slightly above
the energy of the first Lagrange point is the contact connected
sum of two copies of $\RP^3$. In this specific case, however,
one does not need Theorem~\ref{thm:sumEH} to find a periodic
orbit, for the contact form on the
connected sum comes from the standard Weinstein model, so the
existence of a periodic orbit is obvious and indeed known
classically, cf.\ \cite[Section~10.3]{abma78}, \cite[Section~2]{lms85},
\cite[\S 18]{simo71}.

The idea for proving Theorem~\ref{thm:sumGZ} is to study a moduli space
$\WW$ of holomorphic discs inside
the half-symplectisation $(-\infty,0]\times M$ of $(M,\xi)$, with
a Lagrangian boundary condition coming from the explicit model
for the tube defining the connected sum.
If (W$^{\circ}$) does not hold, one can use a deformation of the
evaluation map $\WW\times\D\rightarrow W$ to a map into $M$
to produce a filling of the belt sphere
inside~$M$. This filling can be shown to be a ball under either
of the topological assumptions (i), (ii) in Theorem~\ref{thm:sumGZ}.

A different approach to statements as in that theorem
has been explored in~\cite{gnw16}.
\section{Open books}
\label{section:books}
Besides surgery, another way to construct manifolds are so-called
\emph{open books}.  In order to obtain a manifold $M$ of dimension $2n+1$,
we start with a compact $2n$-dimensional manifold $\Sigma$ with non-empty
boundary $\partial\Sigma$, and a diffeomorphism $\phi$ of~$\Sigma$,
equal to the identity near~$\partial\Sigma$. The mapping cylinder
\[ V(\Sigma,\phi)=\Sigma\times [0,2\pi]/(x,2\pi)\sim(\phi(x),0)\]
has boundary $\partial\Sigma\times S^1$. Form the manifold $M$ by
attaching $\partial\Sigma\times D^2$ along the boundary:
\[ M(\Sigma,\phi)=V(\Sigma,\phi)\cup_{\partial\Sigma\times S^1}
\partial\Sigma\times D^2.\]
The submanifold $\partial\Sigma\times\{0\}\subset M(\Sigma,\phi)$
is called the \emph{binding} of the open book. Its complement
fibres in the obvious way over~$S^1$ with monodromy~$\phi$, and the closures
of the fibres, which are copies of~$\Sigma$, are called the \emph{pages}.
For a beautiful survey on the topology of open books
see the appendix by Winkelnkemper in~\cite{rani98}.

We assume that $\Sigma$ is oriented, the binding carries the
boundary orientation of $\partial\Sigma$, and $M$ the induced orientation.
A positive contact form $\alpha$ on such an open book, i.e.\
a $1$-form satisfying $\alpha\wedge(\rmd\alpha)^n>0$, is said to be
\emph{adapted} if $\alpha$ induces a positive contact form on the
binding, and $\rmd\alpha$ a positive symplectic form on the
interior of each page. A contact structure is said to be
\emph{supported} by the open book if it can be defined by an adapted
contact form.

It is not difficult to construct contact structures supported by
open books, cf.~\cite[Section~7.3]{geig08}. One needs to start
with a \emph{Liouville domain} $(\Sigma,\rmd\beta)$, that is,
an exact symplectic manifold such that the Liouville
vector field $Y$ for $\rmd\beta$ defined by $Y\ip\rmd\beta=\beta$
is transverse to $\partial\Sigma$, pointing outwards.
Moreover, the monodromy diffeomorphism needs to be a symplectomorphism
of $(\Sigma,\rmd\beta)$.
In the sequel, it will be understood that $M(\Sigma,\phi)$
is equipped with the contact structure obtained in this way.
A much deeper theorem of Giroux~\cite{giro02}
says that in fact every contact structure on a closed manifold
is supported by an open book.

In \cite{dgz14}, (W$^{\circ}$) was proved for contact structures
supported by open books under certain assumptions on the binding.
Here is a simplified and more restrictive version of that
result.

\begin{thm}
If $(\Sigma^{2n},\rmd\beta)$ has the structure of a subcritical Stein manifold,
that is, a handle decomposition compatible with the Stein structure
involving only handles up to index $n-1$, then $M(\Sigma,\phi)$
satisfies $\mathrm{(W^{\circ})}$.
\qed
\end{thm}

It is also possible to prove (W$^{\circ}$) for open books when
suitable assumptions are made on the monodromy.

\begin{thm}
\label{thm:monodromy}
Assume that $\phi_1,\ldots,\phi_N$ are symplectomorphisms
of a Liouville manifold $(\Sigma,\rmd\beta)$
with $\phi_1\circ\cdots\circ\phi_N$ symplectically
isotopic to the identity, where the diffeomorphism and isotopies
are equal to the identity near the boundary. Then the
contact manifold $\sqcup_i M(\Sigma,\phi_i)$
satisfies $\mathrm{(W^{\circ})}$.
\end{thm}

\begin{proof}
By \cite[Proposition~8.3]{avde12} or \cite[Theorem~1]{kluk12},
there is a Liouville cobordism from the disjoint union
\[ (M_0,\xi_0):=\bigsqcup_{i=1}^N M(\Sigma,\phi_i)\]
to $(M_1,\xi_1):=M(\Sigma,\phi_1\circ\cdots\circ\phi_N)$, i.e.\
an exact symplectic manifold $(V,\rmd\lambda)$ with oriented
boundary $\partial V=M_1\sqcup -M_0$ such that the Liouville vector field
$Y$ defined by $Y\ip\rmd\lambda=\lambda$ is transverse to the
boundary, pointing into $V$ along $M_0$ (the \emph{concave} end
of the cobordism)
and out of $V$ along~$M_1$ (the \emph{convex} end), and such that
$\ker\lambda|_{TM_i}=\xi_i$.
By the assumption on $\phi_1\circ\cdots\circ\phi_N$ being
symplectically isotopic to the identity, $(M_1,\xi_1)$
is contactomorphic to $M(\Sigma,\id)$.

Topologically we have
\[ M(\Sigma,\id)=\Sigma\times S^1\cup_{\partial\Sigma\times S^1}
\partial\Sigma\times D^2=\partial(\Sigma\times D^2).\]
The obvious Liouville structure
on $\Sigma\times D^2$ (with corners rounded) defines
a Liouville domain with boundary $M(\Sigma,\id)$.

The symplectic completion $\hat{\Sigma}$ of the Liouville
domain $(\Sigma,\rmd\beta)$ is defined by attaching a half-symplectisation
to its boundary:
\[ (\hat{\Sigma},\omega_{\hat{\Sigma}}):=
(\Sigma,\rmd\beta)\cup_{\partial}
\bigl([0,\infty),\rmd(\rme^t\beta|_{T(\partial\Sigma)}\bigr).\]
Then the corners of $\Sigma\times D^2$ can be smoothed inside
the symplectic manifold
\[ \Bigl(\hat{\Sigma}\times\C,\omega_{\hat{\Sigma}}+\frac{r\,\rmd r\wedge\rmd
\theta}{(1+r^2)^2}\Bigr),\]
which can be partially compactified to the symplectic manifold
$\hat{\Sigma}\times\CP^1$.
Finally, we build the symplectic manifold
\[ \bigl((-\infty,0]\times M_0\bigr)\cup_{M_0}V\cup_{M_1}
\bigl((\hat{\Sigma}\times\CP^1)\setminus (\Sigma\times D^2)\bigr),\]
where the first component is again a half-symplectisation.

This manifold contains holomorphic spheres of the form
$\{*\}\times\CP^1$, and one now studies the moduli space
of holomorphic spheres
in the homology class of these standard spheres.
This approach was pioneered by McDuff~\cite{mcdu91}.
Compactness of this moduli space in the sense of symplectic field
theory then leads to the existence of a periodic Reeb orbit
in the concave end $M_0$. The theorem now
follows directly from \cite[Theorem~3.1]{geze12}, where
(W$^{\circ}$) was proved for concave ends of a wide class
of Liouville cobordisms. There, the convex end was supposed
to come from a subcritical Stein manifold; the
necessary modifications for the split Liouville case
considered here were made in the proof of
\cite[Theorem~2.8]{bgz16}.
\end{proof}

\begin{rem}
\label{rem:caps}
(1) If the contact manifolds $M(\Sigma,\phi_i)$ are Liouville
fillable for $i$ in some subset $I\subset\{1,\ldots, N\}$, that is,
if there exists a Liouville cobordism
from the empty set to $M(\Sigma,\phi_i)$ for $i\in I$,
one can cap off these components of $(M_0,\xi_0)$. The preceding
argument then shows that (W$^{\circ}$) holds for $\sqcup_{j\not\in I}
M(\Sigma,\phi_j)$.

(2) Now suppose that instead of having Liouville
fillings we only know that the contact manifolds $M(\Sigma,\phi_i)$, $i\in I$,
are strongly symplectically fillable,
i.e.\ there is a closed symplectic manifold $(W_i,\omega_i)$ with
a Liouville vector field $Y$ defined near the boundary $\partial W_i=M_i$,
pointing outwards and such that $\ker(Y\ip\omega_i)=\xi_i$.
One can then still cap off these components, but now there might be symplectic
spheres in these caps, so the preceding compactness argument
does not go through. This problem can be circumvented with polyfolds.
The symplectic manifold we are considering contains, in the
cap at the convex end, an essential
holomorphic foliation in the sense of \cite[Definition~2.1]{suze17},
see Proposition~3.1 in that paper. As shown in \cite[Corollary~1.3]{suze17},
the concave end $\sqcup_{j\not\in I}M(\Sigma,\phi_j)$ then satisfies the
so-called \emph{strong} Weinstein conjecture: there is
a null-homologous link made up of periodic Reeb orbits.

There no longer needs
to be a contractible periodic Reeb orbit; see \cite[Section~6.4]{geze12}
for an explanation of this fact. Briefly, the compactness argument allows one
to find at least one finite energy plane asymptotic to
a cylinder over a Reeb orbit in $(-\infty,0]\times
\sqcup_{j\not\in I}M(\Sigma,\phi_j)$ or in the part of our
symplectic manifold made up of $V$ and the caps on the $M(\Sigma,\phi_i)$,
$i\in I$. In the case where the caps are Liouville fillings,
this latter part is still a Liouville manifold, and the theorem of Stokes
prevents finite energy planes with negative punctures. Thus, a finite energy
plane exists in $(-\infty,0]\times
\sqcup_{j\not\in I}M(\Sigma,\phi_j)$, and its projection to
$\sqcup_{j\not\in I}M(\Sigma,\phi_j)$ shows that we have a
contractible periodic Reeb orbit.

If the caps are not Liouville fillings, the finite energy plane might sit
in the union of $V$ with the caps, which no longer gives any information
about contractibility of the Reeb orbit inside
$\sqcup_{j\not\in I}M(\Sigma,\phi_j)$. Still, one finds a surface
in $\sqcup_{j\not\in I}M(\Sigma,\phi_j)$ with positive punctures asymptotic to
Reeb orbits, which proves the strong Weinstein conjecture.

(3) The same argument as in (1) or (2) goes through if
the $(M_i,\xi_i)$, $i\in I$, rather than having (Liouville resp.\
strong symplectic) fillings, only admit co-fillings, that is,
if there is a compact manifold (in the respective category) with convex
boundary, where $(M_i,\xi_i)$ is one of several boundary components.
Putting on these `caps' creates new convex boundary components.
These, however, do not affect the compactness argument, since
holomorphic spheres cannot touch such boundaries by the
maximum principle.
\end{rem}

Here is a simple example.

\begin{prop}
Let $\Sigma$ be a compact surface with
boundary and $\phi$ a composition of left-handed Dehn twists.
Then the contact manifold $M(\Sigma,\phi)$ satisfies $\mathrm{(W^{\circ})}$.
In higher dimensions, this is true for $(\Sigma,\rmd\beta)$ a Liouville
manifold and $\phi$ a composition of left-handed Dehn twists along
Lagrangian spheres.
\end{prop}

\begin{proof}
When $\Sigma$ is a surface, we apply the preceding discussion to
\[ (M_0,\xi_0):= M(\Sigma,\phi)\sqcup M(\Sigma,\phi^{-1}).\]
By a result of Loi--Piergallini and Giroux, see~\cite{geig12} for an
exposition, the contact manifold $M(\Sigma,\phi^{-1})$, where the monodromy
consists of right-handed Dehn twists, is Stein fillable, which means in
particular Liouville fillable. Now apply Remark~\ref{rem:caps}~(1).
The higher-dimensional analogue of the result by Loi--Piergallini
and Giroux is proved in \cite{koer10}; then the argument is
completely analogous.
\end{proof}

\begin{rem}
In the case where $M(\Sigma,\phi)$ is $3$-dimensional,
and if the Dehn twists on the surface $\Sigma$ making up $\phi$
are along homotopically non-trivial
curves, one can argue alternatively as follows.
By a result of Y{\i}lmaz~\cite{yilm11}, in this situation the contact
$3$-manifold $M(\Sigma,\phi)$ is overtwisted. Then apply Hofer's
result~\cite{hofe93} mentioned in Section~\ref{section:surgery}.
\end{rem}

One can also turn the argument on its head, as it were, and derive
statements about non-existence of fillings.

\begin{ex}
Let $M(\Sigma,\phi)$ be a contact manifold whose contact structure
can be defined by a contact form without contractible periodic Reeb
orbits, such as that described in the introduction. Then
the contact manifold $M(\Sigma,\phi^{-1})$ is not Liouville
fillable.
\end{ex}

From the Riemann--Finsler perspective, one is of course primarily
interested in proving the
Weinstein conjecture for unit (co-)tangent bundles.
However, we should iterate our remark from Section~\ref{section:dual}
that the principal aim in interpreting Finsler geodesic flows as Reeb flows
is not to reprove theorems in Finsler geometry, but to see
them as instances of results in a more general theory.

Hofer and Viterbo \cite{hovi88} have confirmed the Weinstein conjecture
for compact connected hypersurface $M\subset T^*Q$ in cotangent bundles
that satisfy the following conditions:
\begin{itemize}
\item[(i)] $M$ is of contact type, i.e.\ there is a Liouville
vector field $Y$ ($L_Y\ocan=\ocan$) defined near and transverse to~$M$.
\item[(ii)] The bounded component of $T^*Q\setminus M$ contains
the zero section~$Q$.
\end{itemize}
By Section~\ref{section:Finsler}, this
result includes the classical theorem of Lyusternik
and Fet~\cite{lyfe51} on the existence of a closed geodesic
on any compact Riemannian manifold.

The methods discussed in this section lead to
a proof of the strong Weinstein conjecture for cotangent bundles
of manifolds of the form $Q\times S^1$, where $Q$ is any closed manifold,
see~\cite[Corollary~4.8]{geze12}.

Liouville structures on manifolds of the form $Q\times [-1,1]$, with $Q$ a
compact left-quotient of one of the $3$-dimensional Lie groups
$\mathrm{Sol}^3$ or $\widetilde{\mathrm{SL}}_2$, were
constructed in~\cite{geig95,mcdu91}. Then the contact manifold
$M(Q\times[-1,1],\id)$, which satisfies (W$^{\circ}$) by
Theorem~\ref{thm:monodromy}, is diffeomorphic to
$\partial(Q\times [-1,1]\times D^2)=Q\times S^2$, so at least it
has the topology of the unit cotangent bundle of~$Q$. However, there
is no obvious identification of the contact structure on
$M(Q\times[-1,1],\id)$ with the canonical structure on the 
unit cotangent bundle.

Further examples of unit cotangent bundles for which the strong Weinstein
conjecture holds are given in ~\cite[Section~3.3]{suze17}.
\begin{ack}
We thank Umberto Hryniewicz for his perspicacious
comments on a draft version of this paper. We also thank
Felix Schlenk and the referees for comments on the
Finsler geometry of cotangent bundles that have resulted in the
writing of Section~\ref{section:dual}.
\end{ack}


\begin{thebibliography}{10}
%
\bibitem{abpa94}
{\sc M. Abate and G. Patrizio},
\textit{Finsler Metrics -- A Global Approach},
Lecture Notes in Mathematics \textbf{1591}
(Springer-Verlag, Berlin, 1994).
%
\bibitem{abma78}
{\sc R. Abraham and J. E. Marsden},
\textit{Foundations of Mechanics},
2nd edition (Benjamin Cummings, Reading, MA, 1978).
%
\bibitem{afkp12}
{\sc P. Albers, U. Frauenfelder, O. van Koert and G. Paternain},
Contact geometry of the restricted three-body problem,
\textit{Comm. Pure Appl. Math.}
\textbf{65} (2012), 229--263.
%
\bibitem{avde12}
{\sc R. Avdek},
Liouville hypersurfaces and connect sum cobordisms,
\texttt{arXiv:1204.3145}.
%
\bibitem{baro12}
{\sc V. Bangert and N. R\"ottgen},
Isoperimetric inequalities for minimal submanifolds in
Riemannian manifolds: a counterexample in higher codimension,
\textit{Calc. Var. Partial Differential Equations}
\textbf{45} (2012), 455--466.
%
\bibitem{bgz16}
{\sc K. Barth, H. Geiges and K. Zehmisch},
The diffeomorphism type of symplectic fillings,
\texttt{arXiv:1607.03310}.
%
\bibitem{bart12}
{\sc T. Barthelm\'e},
A new Laplace operator in Finsler geometry and periodic orbits
of Anosov flows,
Ph.D. thesis, Strasbourg (2012);
\texttt{arXiv:1204.0879}.
%
\bibitem{bpv09}
{\sc J. B. van den Berg, F. Pasquotto and R. C. Vandervorst},
Closed characteristics on non-compact hypersurfaces in~$\R^{2n}$,
\textit{Math. Ann.}
\textbf{343} (2009), 247--284.
%
\bibitem{bprv}
{\sc J. B. van den Berg, F. Pasquotto, T. Rot and
R. C. A. M. Vandervorst},
On periodic orbits in cotangent bundles of non-compact manifolds,
\textit{J. Symplectic Geom.}
\textbf{14} (2016), 1145--1173.
%
\bibitem{buse55}
{\sc H. Busemann},
\textit{The Geometry of Geodesics}
(Academic Press, New York, 1955).
%
\bibitem{ccl99}
{\sc S. S. Chern, W. H. Chen and K. S. Lam},
\textit{Lectures on Differential Geometry},
Series on University Mathematics \textbf{1}
(World Scientific, River Edge, NJ, 1999).
%
\bibitem{dgz14}
{\sc M. D\"orner, H. Geiges and K. Zehmisch},
Open books and the Weinstein conjecture,
\textit{Q. J. Math.}
\textbf{65} (2014), 869--885.
%
\bibitem{elia90}
{\sc Ya. Eliashberg},
Filling by holomorphic discs and its applications,
\textit{Geometry of Low-Dimensional Manifolds, Vol.~2\/} (Durham, 1989),
London Mathematical Society Lecture Note Series \textbf{151}
(Cambridge University Press, Cambridge, 1990), 45--67.
%
\bibitem{elia92}
{\sc Ya. Eliashberg},
Contact $3$-manifolds twenty years since J.~Martinet's work,
\textit{Ann. Inst. Fourier (Grenoble)\/}
\textbf{42} (1992), 165--192.
%
\bibitem{elho94}
{\sc Ya. Eliashberg and H. Hofer},
A Hamiltonian characterization of the three-ball,
\textit{Differential Integral Equations}
\textbf{7} (1994), 1303--1324.
%
\bibitem{fls15}
{\sc U. Frauenfelder, C. Labrousse and F. Schlenk},
Slow volume growth for Reeb flows on spherizations and contact
Bott--Samelson theorems,
\textit{J. Topol. Anal.}
\textbf{7} (2015), 407--451.
%
\bibitem{geig95}
{\sc H. Geiges},
Examples of symplectic $4$-manifolds with disconnected boundary
of contact type,
\textit{Bull. London Math. Soc.}
\textbf{27} (1995), 278--280.
%
\bibitem{geig08}
{\sc H. Geiges},
\textit{An Introduction to Contact Topology},
Cambridge Studies in Advanced Mathematics \textbf{109}
(Cambridge University Press, Cambridge, 2008).
%
\bibitem{geig12}
{\sc H. Geiges},
Contact structures and geometric topology,
\textit{Global Differential Geometry},
Springer Proceedings in Mathematics \textbf{17}
(Springer-Verlag, Berlin, 2012), 463--489.
%
\bibitem{geig16}
{\sc H. Geiges},
\textit{The Geometry of Celestial Mechanics},
London Mathematical Society Student Texts \textbf{83}
(Cambridge University Press, Cambridge, 2016).
%
\bibitem{grz14}
{\sc H. Geiges, N. R\"ottgen and K. Zehmisch},
Trapped Reeb orbits do not imply periodic ones,
\textit{Invent. Math.}
\textbf{198} (2014), 211--217.
%
\bibitem{geze12}
{\sc H. Geiges and K. Zehmisch},
Symplectic cobordisms and the strong Weinstein conjecture,
\textit{Math. Proc. Cambridge Philos. Soc.}
\textbf{153} (2012), 261--279.
%
\bibitem{geze16}
{\sc H. Geiges and K. Zehmisch},
Reeb dynamics detects odd balls,
\textit{Ann. Sc. Norm. Super. Pisa Cl. Sci. (5)}\/
\textbf{15} (2016), 663--681.
%
\bibitem{geze16a}
{\sc H. Geiges and K. Zehmisch},
The Weinstein conjecture for connected sums,
\textit{Int. Math. Res. Not. IMRN}
\textbf{2016} (2016), 325--342.
%
\bibitem{gnw16}
{\sc P. Ghiggini, K. Niederkr\"uger and C. Wendl},
Subcritical contact surgeries and the topology of symplectic fillings,
\textit{J. \'Ec. polytech. Math.}
\textbf{3} (2016), 163--208.
%
\bibitem{giro02}
{\sc E. Giroux},
G\'eom\'etrie de contact: de la dimension trois vers les dimensions
sup\'erieures,
\emph{Proceedings of the International Congress of Mathematicians,
Vol.~II} (Higher Education Press, Beijing, 2002),
405--414.
%
\bibitem{grom85}
{\sc M. Gromov},
Pseudoholomorphic curves in symplectic manifolds,
{\it Invent. Math.}
{\bf 82} (1985), 307--347.
%
\bibitem{hofe93}
{\sc H. Hofer},
Pseudoholomorphic curves in symplectizations,
\textit{Invent. Math.}
\textbf{114} (1993), 515--563.
%
\bibitem{hovi88}
{\sc H. Hofer and C. Viterbo},
The Weinstein conjecture in cotangent bundles and related results,
\textit{Ann. Scuola Norm. Sup. Pisa Cl. Sci. (4)}\/
\textbf{15} (1988), 411--445.
%
\bibitem{hrsa13}
{\sc U. L. Hryniewicz and P. A. S. Salom\~{a}o},
Global properties of tight Reeb flows with applications
to Finsler geodesic flows on~$S^2$,
\textit{Math. Proc. Cambridge Philos. Soc.}
\textbf{154} (2013), 1--27.
%
\bibitem{hrsa13a}
{\sc U. L. Hryniewicz and P. A. S. Salom\~{a}o},
\textit{Introdu\c{c}\~{a}o \`{a} Geometria Finsler},
Publica\c{c}\~{o}es Matem\'{a}ticas do IMPA
(Rio de Janeiro, 2013).
%
\bibitem{hutc10}
{\sc M. Hutchings},
Taubes's proof of the Weinstein conjecture in dimension three,
\textit{Bull. Amer. Math. Soc. (N.S.)\/}
\textbf{47} (2010), 73--125.
%
\bibitem{kluk12}
{\sc M. Klukas},
Open books and exact symplectic cobordisms,
\texttt{arXiv:1207.5647}.
%
\bibitem{koer10}
{\sc O. van Koert},
Lecture notes on stabilization of contact open books,
\texttt{arXiv:1012.4359}.
%
\bibitem{lima87}
{\sc P. Libermann and C.-M. Marle},
\textit{Symplectic Geometry and Analytical Mechanics},
Mathematics and its Applications \textbf{35}
(Reidel Publishing Co., Dordrecht, 1987).
%
\bibitem{lms85}
{\sc J. Llibre, R. Mart\'{\i}nez and C. Sim\'{o}},
Transversality of the invariant manifolds associated to
the Lyapunov family of periodic orbits near $L_2$ in the
restricted three-body problem,
\textit{J. Differential Equations}
\textbf{58} (1985), 104--156.
%
\bibitem{lyfe51}
{\sc L. A. Lyusternik and A. I. Fet},
Variational problems on closed manifolds (Russian),
\textit{Doklady Akad. Nauk SSSR (N.S.)}\/
\textbf{81} (1951), 17--18.
%
\bibitem{mcdu91}
{\sc D. McDuff},
Symplectic manifolds with contact type boundaries,
\textit{Invent. Math.}
\textbf{103} (1991), 651--671.
%
\bibitem{rade04a}
{\sc H.-B. Rademacher},
A sphere theorem for non-reversible Finsler metrics,
\textit{Math. Ann.}
\textbf{328} (2004), 373--387.
%
\bibitem{rade04b}
{\sc H.-B. Rademacher},
Nonreversible Finsler metrics of positive flag curvature,
\textit{A Sampler of Riemann--Finsler Geometry},
Mathematical Sciences Research Institute Publications \textbf{50}
(Cambridge University Press, Cambridge, 2004), 261--302.
%
\bibitem{rani98}
{\sc A. Ranicki},
\textit{High-dimensional Knot Theory},
Springer Monographs in Mathematics
(Springer-Verlag, Berlin, 1998).
%
\bibitem{schl16}
{\sc F. Schlenk},
Why Finsler flows naturally live in cotangent bundles,
private communication, December 2016.
%
\bibitem{simo71}
{\sc C. L. Siegel and J. K. Moser},
\textit{Lectures on Celestial Mechanics},
Die Grundlehren der mathematischen Wissenschaften \textbf{187}
(Springer-Verlag, Heidelberg, 1971).
%
\bibitem{suze16}
{\sc S. Suhr and K. Zehmisch},
Linking and closed orbits,
\textit{Abh. Math. Semin. Univ. Hambg.}
\textbf{86} (2016), 133--150.
%
\bibitem{suze17}
{\sc S. Suhr and K. Zehmisch},
Polyfolds, cobordisms, and the strong Weinstein conjecture,
\textit{Adv. Math.}
\textbf{305} (2017), 1250--1267.
%
\bibitem{taub07}
{\sc C. H. Taubes},
The Seiberg--Witten equations and the Weinstein conjecture,
\textit{Geom. Topol.}
\textbf{11} (2007), 2117--2202.
%
\bibitem{wein79}
{\sc A. Weinstein},
On the hypotheses of Rabinowitz' periodic orbit theorems,
\textit{J. Differential Equations\/}
\textbf{33} (1979), 336--352.
%
\bibitem{wein91}
{\sc A. Weinstein},
Contact surgery and symplectic handlebodies,
\textit{Hokkaido Math. J.}
\textbf{20} (1991), 241--251.
%
\bibitem{yilm11}
{\sc E. Y{\i}lmaz},
A note on overtwisted contact structures,
\textit{Studia Sci. Math. Hungar.}
\textbf{48} (2011), 130--134.
%
\end{thebibliography}
\end{document}